\documentclass[10pt]{amsart}
\usepackage[margin=1in]{geometry}
\usepackage{amscd,amsmath,amsxtra,amsthm,amssymb,stmaryrd,xr,mathrsfs,mathtools,enumerate,commath, comment}
\usepackage{stmaryrd}
\usepackage[dvipsnames]{xcolor}
\usepackage{commath}
\usepackage{dsfont}
\usepackage{comment}
\usepackage{tikz-cd}
\usepackage{longtable} 
\usepackage{pdflscape} 
\usepackage{booktabs}
\usepackage{hyperref}
\definecolor{vegasgold}{rgb}{0.77, 0.7, 0.35}
\definecolor{darkgoldenrod}{rgb}{0.72, 0.53, 0.04}
\definecolor{gold(metallic)}{rgb}{0.83, 0.69, 0.22}
\hypersetup{
 colorlinks=true,
 linkcolor=OliveGreen,
 citecolor=NavyBlue,
 }

\usepackage[all,cmtip]{xy}

\DeclareFontFamily{U}{wncy}{}
\DeclareFontShape{U}{wncy}{m}{n}{<->wncyr10}{}
\DeclareSymbolFont{mcy}{U}{wncy}{m}{n}
\DeclareMathSymbol{\Sh}{\mathord}{mcy}{"58}
\usepackage[T2A,T1]{fontenc}
\usepackage[OT2,T1]{fontenc}

\newtheorem{theorem}{Theorem}[section]
\newtheorem{Lemma}[theorem]{Lemma}
\newtheorem*{theorem*}{Theorem}

\newtheorem{proposition}[theorem]{Proposition}

\newtheorem{definition}[theorem]{Definition}

\newtheorem*{ass*}{Assumption}

\newtheorem{remark}[theorem]{Remark}

\newcommand{\tr}{\operatorname{tr}}

\newcommand{\Q}{\mathbb{Q}}
\newcommand{\F}{\mathbb{F}}

\newcommand{\cO}{\mathcal{O}}

\newcommand{\op}[1]{\operatorname{#1}}

\newcommand{\bx}{\textbf{x}}
\newcommand{\bz}{\textbf{z}}
\newcommand{\bh}{\textbf{h}}
\newcommand{\by}{\textbf{y}}
\newcommand{\be}{\textbf{e}}
\newcommand{\ba}{\textbf{a}}
\newcommand{\E}{\mathbb E}
\numberwithin{equation}{section}

\begin{document}

\title[Projective spaces]{Morphism spaces on low degree hypersurfaces}

\author[H. Mishra]{Hrishabh Mishra}
\address[]{University of Wisconsin-Madison, WI, USA.}
\email{hmishra4@wisc.edu}

\date{August 13, 2026. \emph{Revised:} September, 2026.}

\begin{abstract}
 For $r> 2$, we study the moduli space parameterising fixed degree morphisms $\mathbb P^r\to X$ where $X$ is a smooth hypersurface of low degree. More precisely, we prove the following result: let $n\geq 2, e\geq 1$ and $X\subset \mathbb P^{n-1}$ a smooth degree $d\geq 2$ hypersurface over an algebraically closed field of characteristic zero or greater than $d$, then $\mathrm{Mor}_e(\mathbb P^r, X)$ is irreducible of the expected dimension if
 \[
 n> 2^d(d-1)\binom{de+r-1}{r-1}.
 \]
 Our result extends the result of Browning-Yamagishi for $r=2$ via multiblock Weyl differencing.
\end{abstract}

\maketitle
\vspace{-1.5em}
{
 \hypersetup{linkcolor=black}
 \tableofcontents
}
\vspace{-3em}

\section{Introduction}

Let $K$ be an algebraically closed field, and let
\[
 X=\{f=0\}\subset \mathbb P_K^{n-1}
\]
be a smooth hypersurface of degree $d\geq 2$. For $r, e\geq 1$, we write $\op{Mor}_e(\mathbb P^r,X)$ for the scheme parametrising morphisms
$g:\mathbb P^r\to X$ satisfying
\[
 g^*\cO_X(1)\simeq \cO_{\mathbb P^r}(e),
\]
hence degree $e$ morphisms. Such a morphism is represented by an $n$-tuple of degree $e$ homogeneous forms in $r+1$ variables, without a common zero, up to simultaneous scalar multiplication. The condition that the image lie in $X$ is the identity $f(g_1,\ldots,g_n)=0$, which gives $\binom{de+r}{r}$ scalar equations. Thus the expected dimension is
\begin{equation}\label{eq:expected-dim}
 \mu_r(e)=n\binom{e+r}{r}-\binom{de+r}{r}-1.
\end{equation}
Our main result is the following.

\begin{theorem}\label{thm:main}
Let $r\geq 2$, $d\geq 2$, and $e\geq 1$. Suppose $K$ has characteristic zero or greater than $d$. Let $X\subset \mathbb P_K^{n-1}$ be a smooth hypersurface of degree $d$. If
\[
 n>2^d(d-1)\binom{de+r-1}{r-1},
\]
then $\op{Mor}_e(\mathbb P^r,X)$ is irreducible and has dimension $\mu_r(e)$.
\end{theorem}

Determining the nonemptiness, dimension and irreducible components of moduli spaces of maps is a basic problem in the geometry of rational curves. When the source is $\mathbb P^1$, one may study either the space of parametrised morphisms or its compactification by Kontsevich stable maps. For general low degree hypersurfaces, Harris--Roth--Starr proved irreducibility, reducedness and the expected dimension property for the genus zero Kontsevich spaces in a suitable range, and Riedl--Yang later obtained a substantially sharper result for general Fano hypersurfaces (\cite{HarrisRothStarr2004}, \cite{RiedlYang2019}). For arbitrary smooth hypersurfaces, Browning and Vishe introduced a different approach based on the function field circle method, proving irreducibility and the expected dimension property in a uniform low degree range \cite{Browning2017}. Although the expected dimension is readily predicted by deformation theory, proving that it is the actual dimension, and that no unexpected irreducible components occur, requires substantial global information about the space of maps. Function field analytic number theory has also been used to establish finer properties of relevant moduli spaces, for example, in \cite{haseliu2025highergenuscirclemethod} Hase-Liu reinterprets the Browning-Vishe machinery geometrically and proves the irreducibility and expected dimension property for moduli space of arbitrary genus curves on smooth low degree hypersurfaces, Browning--Sawin \cite{Browning2023} establish bounds on the dimension of the singular locus, and Glas--Hase-Liu \cite{glas2024terminalsingularitiesmodulispace} used analytic methods to study singularities of the relevant moduli spaces.

The study of higher dimensional rational subvarieties is a natural extension of the theory of rational curves, but it reflects rather stronger geometric properties of the target. Such parameter spaces provide a natural framework for studying the existence and variation of higher dimensional rational subvarieties on $X$. This perspective is related to the theory of higher Fano manifolds, where positivity of higher Chern characters is expected to force the existence of higher dimensional rational subvarieties. For example, de Jong and Starr showed that, under suitable positivity and irreducibility hypotheses, a Fano variety is covered by rational surfaces \cite{deJongStarr2007}.

Even for hypersurfaces, relatively few global results are known for parameter spaces of higher dimensional rational varieties. Browning and Yamagishi recently treated the first case involving all morphisms from a fixed higher dimensional source, proving irreducibility and the expected dimension for $\op{Mor}_e(\mathbb P^2,X)$ when $X$ is an arbitrary smooth hypersurface of sufficiently low degree \cite{browning2025rationalsurfaceslowdegree}. Theorem~\ref{thm:main} extends the surface theorem of Browning--Yamagishi~\cite{browning2025rationalsurfaceslowdegree}
to projective spaces of arbitrary dimension $r> 2$, giving a uniform result for higher dimensional projective sources. In forthcoming work, Matthew Hase-Liu treats the more general case of morphism spaces from an arbitrary projective source, obtaining irreducibility and the expected dimension in a range quadratic in the degree by different methods, based on his quadratic Birch theorem over function fields \cite{HL26}.

It is also natural to consider the \emph{projective coefficient compactification}
\begin{equation} \label{eq:def-Qe}
Q_e(X) := \left\{ [g_1:\cdots:g_n] \in \mathbb P\!\left(H^0\bigl(\mathbb P^r,\mathcal O_{\mathbb P^r}(e)\bigr)^{\oplus n} \right) : f(g_1,\ldots,g_n)=0 \right\}.
\end{equation}
The morphism scheme is the open locus
\[
\operatorname{Mor}_e(\mathbb P^r,X)\subset Q_e(X)
\]
where the forms $g_1,\ldots,g_n$ have no common zero. Thus, $Q_e(X)$ also contains degenerate tuples with base points. Proving that $Q_e(X)$ is irreducible is therefore stronger than proving irreducibility of the morphism locus: it shows that the base point boundary does not support an additional irreducible component.

The formulation in terms of $Q_e(X)$ is particularly natural for our method. The function field counting problem used in the proof counts all coefficient tuples satisfying $f(g_1,\ldots,g_n)=0$, including those with base points, and hence counts the affine cone over $Q_e(X)$ rather than just the cone over the open morphism locus. Thus, our counting argument naturally proves the stronger result for $Q_e(X)$.

\begin{theorem}\label{thm:Q_e}
    Under the hypothesis of Theorem \ref{thm:main}, the projective coefficient space $Q_e(X)$ is irreducible of dimension $\mu_r(e)$.
\end{theorem}

The specialization $e=1$ connects our result with the classical geometry of Fano schemes. Indeed, if $F_r(X)$ denotes the Fano scheme of $r$-planes contained in $X$, then forgetting the parametrization gives a principal $\operatorname{PGL}_{r+1}$ bundle
\[
\operatorname{Mor}_1(\mathbb P^r,X) \longrightarrow F_r(X).
\]
Consequently, our theorem implies that $F_r(X)$ is geometrically irreducible of the expected dimension
\[
(r+1)(n-r-1)-\binom{d+r}{r}.
\]

Fano schemes of linear spaces on general hypersurfaces and complete intersections have been studied extensively by Borcea, Debarre--Manivel and Hochster--Laksov \cite{Borcea1990,DebarreManivel1998,HochsterLaksov1987}; see also the work of Starr on linear spaces and Veronese varieties contained in sufficiently general hypersurfaces \cite{Starr2017}. For arbitrary smooth hypersurfaces in characteristic zero, Beheshti and Riedl proved irreducibility and the expected-dimension property in a substantially sharper range \cite{BeheshtiRiedl2021}.

Thus, in the linear case, the novelty of our result is not an improvement of the characteristic zero numerical range. Rather, to
the best of our knowledge, it gives the first uniform positive characteristic analogue for $r$ planes with arbitrary
$r$, applying to every smooth hypersurface over an algebraically closed field of characteristic zero or greater than $d$. More importantly, the Fano scheme statement is only the degree one specialization of our results for higher degree morphisms and their coefficient compactifications.

We briefly describe the main difficulty in proving Theorem~\ref{thm:main}. After passing to a finite field and choosing one distinguished affine coordinate $u$ on $\mathbb P^r$, a degree $e$ polynomial can be written as
\[
g_i(u,\mathbf z)= \sum_{\mathbf a\in\Delta_{r-1}(e)} t_{i,\mathbf a}(u)\mathbf z^{\mathbf a}, \qquad \deg t_{i,\mathbf a}\leq e-|\mathbf a|,
\]
where
\[
\Delta_{r-1}(e) = \left\{\mathbf a\in\mathbb Z_{\geq 0}^{r-1}: |\mathbf a|\leq e \right\}.
\]
The coefficient blocks $t_{i,\mathbf a}$ therefore range over boxes whose side lengths vary from $q$ to $q^{e+1}$. Expanding $f(g_1,\ldots,g_n)$ in the variables $\mathbf z$ produces
\[
\binom{de+r-1}{r-1}
\]
coefficient polynomials. The resulting system has a large joint singular locus and is supported on a highly lopsided box, so standard circle method estimates for systems of forms do not apply directly with the strength required here.

Browning and Yamagishi overcame these difficulties for $r=2$ by choosing, for each coefficient equation, a Weyl differencing pattern which isolates that equation from all the others. In their setting the coefficient blocks are indexed by an interval and the relevant
leading terms involve at most two different blocks. For general $r$, the indices instead form an $(r-1)$ dimensional simplex, and
a single coefficient may involve an arbitrary number of distinct blocks. Schindler tackles similar issues of lopsidedness over $\Q$ in \cite{Schindler2015}.

The main new ingredient of this paper is a selective multiblock Weyl differencing estimate adapted to this situation. For each coefficient index, we construct a balanced decomposition which selects the appropriate coefficient blocks and their multiplicities. After freezing the unselected blocks, the differencing process isolates the chosen coefficient frequency uniformly in all the other frequencies and in the frozen coefficients. A geometric mean over the resulting one frequency estimates then reduces the full multidimensional circle integral to a product of one dimensional mean values. This gives the required point counting estimate with leading constant one. Lang--Weil then yields irreducibility over finite fields, and the general result follows by spreading out. The idea of exploiting the common multilinear structure among the coefficient equations, rather than treating them as a general system of forms, also appears in Brandes’s work \cite{Brandes2013}.

The multiblock separation principle developed here may also be useful for other systems whose individual equations can be isolated through
different block multidegrees, even when the full system has a large singular locus. In particular, it suggests a route towards analogous
questions for morphisms into smooth complete intersections---the present argument separates the source coefficient indices, while a
system estimate would be required to handle the remaining frequencies coming from the different equations defining the intersection. In the recent paper \cite{browning2024rationalcurvescompleteintersections} Browning--Vishe--Yamagishi study the space of rational curves on smooth complete intersections adapting the work of Rydin Myerson \cite{RydinMyerson2017, RydinMyerson2018} to the setting of function fields.

\subsection*{Organization} In \S\ref{sec:preliminaries} we introduce the function field notation and the auxiliary point counting estimates. In \S\ref{sec:weyl-differencing} we establish the selective multiblock Weyl differencing inequality, and in
\S\ref{sec:mean-value} we prove the associated mean value estimate. We complete the finite field count in \S\ref{sec:finite-field-count}, and in \S\ref{sec:geometric-deduction} we deduce Theorem~\ref{thm:main} using Lang--Weil and spreading out.

\subsection*{Acknowledgements} The author is very grateful to Tim Browning for many helpful conversations, encouragement, a careful reading of an earlier version of this paper, and numerous comments and suggestions. The author also thanks Timothy Cheek for carefully reading an earlier version and for helpful feedback, Jordan Ellenberg for helpful discussions and encouragement, and Matthew Hase-Liu for apprising the author of his upcoming work and for sharing a preliminary version.

\subsubsection*{AI Statement} After a preliminary version of this manuscript was prepared, ChatGPT was used for assistance in polishing parts of the exposition, minor LaTeX formatting, and bibliographic formatting.

\section{Background and auxiliary lemmas}\label{sec:preliminaries}

Let $\F_q$ denote a finite field with characteristic $p> d$. As usual put
\[
\mathbb K= \F_q(u),\, \mathbb K_\infty=\F_q((u^{-1})).
\]
Recall the absolute value $|\cdot|$ on $\mathbb K$ given by
\[
|a|=\begin{cases}
 q^{\deg a}\, &\mathrm{if}\, a \in \mathbb K^\times,\\
 0\, &\mathrm{otherwise.}
\end{cases}
\]
We extend this absolute value to $\mathbb K_\infty$ and identify it as the completion of $\mathbb K$ with respect to $|\cdot|$. Set
\[
\mathbb T=\{\alpha\in\mathbb K_\infty: |\alpha|< 1\}=\left\{\sum_{i<0}a_iu^i: a_i\in \F_q\right\},
\]
equipped with the normalized Haar measure $d\alpha$. Let $\psi: \mathbb K_\infty \to \mathbb C^\times$ denote the nontrivial unitary additive character induced by the character $e_q: \F_q\to \mathbb C^\times$ defined as
\[
e_q(a):= \exp(2\pi i\tr(a)/p),
\]
where $\tr: \F_q \to \F_p$ denotes the trace map. Concretely, if $\alpha = \sum_{i<M}a_iu^i \in \mathbb K_\infty$, then we set
\[
\psi(\alpha)= e_q(a_{-1}).
\]
For any $a \in \F_q[u]$,
\begin{equation}\label{int-orth-id}
 \int_{\mathbb T}\psi(\alpha a)d\alpha=
 \begin{cases}
  0,\, &\text{if}\, a\neq0,\\
  1,\, &\text{otherwise}.
 \end{cases}
\end{equation}
We also have the finite box orthogonality identity

\begin{equation}\label{sum-orth-id}
 \sum_{x\in \F_q[u], |x|< q^P} \psi(\beta x)=
 \begin{cases}
  q^P,\, &\text{if}\, \|\beta\|< q^{-P},\\
  0,\, &\text{otherwise}.
 \end{cases}
\end{equation}

Here we denote by $\|\alpha\|$ the absolute value of fractional part of $\alpha$. For a vector $\textbf{x}$ we set $|\textbf{x}|$ to be the maximum of the absolute values of its coordinates. For an integer $P> 0$, set
\[
\mathcal B(P):= \{\textbf{x}\in \F_q[u]^n: |\textbf{x}|< q^P\},
\]
hence, $\#\mathcal B(P)= q^{nP}$. We now state two crucial lemmas: a function field version of Davenport's shrinking lemma and a uniform point counting lemma.

\begin{Lemma}[Lemma 4.2, \cite{browning2025rationalsurfaceslowdegree}]\label{lem:shrinking}
Let $L_1,\ldots,L_N\in \mathbb K_\infty[t_1,\ldots,t_N]$ be linear forms
\[
 L_i(\mathbf t)=\sum_{j=1}^N\gamma_{ij}t_j
\]
whose coefficient matrix is symmetric. Let $A\in\mathbb Q_{\geq 0}$ and $Z\in\mathbb Q_{\leq 0}$, and put
\[
 \mathcal N_A(Z)=\#\left\{\mathbf t\in\F_q[u]^N:
 |t_i|<q^{A+Z},\quad \|L_i(\mathbf t)\|<q^{-A+Z}
 \ \text{for every }i\right\}.
\]
If $Z_1\leq Z_2\leq 0$, $A-Z_2\geq 0$, and $A\pm Z_i$ are integers, then
\[
 \mathcal N_A(Z_2)\leq q^{N(Z_2-Z_1)}\mathcal N_A(Z_1).
\]
\end{Lemma}

The next lemma gives uniform point count on affine varieties.

\begin{Lemma}[Lemma 2.1, \cite{browning2025rationalsurfaceslowdegree}]\label{lem:poly-points}
Let $V\subset \mathbb A^N_{\F_q}$ be an affine variety of dimension $m$. For every integer $J\geq 1$,
\[
 \#\{\mathbf x\in V(\F_q[u]):|\mathbf x|<q^J\} \leq \delta(V)\cdot q^{Jm},
\]
where $\delta(V)= \sum_i\deg V_i$ with $V=\cup_i V_i$ is the decomposition of $V$ into its irreducible components.
\end{Lemma}

We shall also use the following elementary form of Dirichlet approximation, we include the proof as well.

\begin{Lemma}[Dirichlet approximation]\label{lem:dirichlet}
Let $L\geq 0$ and let $\alpha\in\mathbb T$. There exist polynomials $a,g\in\F_q[u]$, with $g$ monic and $\gcd(a,g)=1$, such that
\[
 0<|g|\leq q^L,\qquad |g\alpha-a|<q^{-L}.
\]
\end{Lemma}

\begin{proof}
Consider the $q^{L+1}$ fractional parts
\[
\{\{h\alpha\}: \deg h\leq L\}
\]
of $h\alpha$. Partition $\mathbb T$ according to the coefficients of $u^{-1},\ldots,u^{-L}$. Since two fractional parts lie in the same part, their difference gives nonzero $g$ with $\deg g\leq L$ and $\|g\alpha\|<q^{-L}$. We take $a$ to be the polynomial part of $g\alpha$
and canceling the common divisor of $a$ and $g$ proves the result.
\end{proof}

\section{Selective multiblock Weyl differencing}\label{sec:weyl-differencing}

Let $f\in \F_q[x_1,\cdots, x_n]$ be a nonsingular homogeneous form of degree $d$. There is a unique symmetric $d$-linear form
\[
\Gamma_f: (\F_q^n)^d\to \F_q
\]
with
\[
f(\textbf{x})= \Gamma_f(\textbf{x}, \dots, \textbf{x}).
\]
Note that
\begin{equation}\label{eq:gradient-polarisation}
 \frac{\partial f}{\partial x_i}(\mathbf x)
 =d\,\Gamma_f(\mathbf x,\ldots,\mathbf x,\mathbf e_i).
\end{equation}
Fix positive integers $m_1, \dots, m_s$ such that $m_1+\cdots+m_s=d$ and $P_1, \dots, P_s$. We require that the variable $\textbf{x}_j$ ranges over $B_j= \mathcal B(P_j)$ and $M_j= \#B_j$. Set $M= \prod_j M_j$. Let $\Omega$ be an indexing set; suppose we have a family of polynomials over $\mathbb K_\infty$
\[
R_{\alpha, \omega}(\textbf{x}_1, \dots, \textbf{x}_s)\,\, \mathrm{ for }\,\, (\alpha, \omega) \in \mathbb T\times \Omega.
\]
\begin{definition}\label{def:deficiency}
A monomial $\mathfrak m$ in the scalar coordinates of $\bx_1,\ldots,\bx_s$ is called \emph{$(m_1,\ldots,m_s)$ deficient} if there exists at least one block index $j$ such that
\[
 \deg_{\bx_j}\mathfrak m<m_j,
\]
where $\deg_{\bx_j}$ is the total degree in the $n$ scalar coordinates of $\bx_j$.
\end{definition}

Assume that every monomial of every $R_{\alpha,\omega}$ is $(m_1,\ldots,m_s)$ deficient. Let $c\in\F_q^\times$, and define the multihomogeneous leading form
\begin{equation}\label{eq:G}
 G(\bx_1,\ldots,\bx_s)=c\,\Gamma_f(\underbrace{\bx_1,\ldots,\bx_1}_{m_1\text{ times}},\ldots, \underbrace{\bx_s,\ldots,\bx_s}_{m_s\text{times}}).
\end{equation}
The full phase is
\begin{equation}\label{eq:uniform-phase}
 \Phi_{\alpha,\omega}(\bx_1,\ldots,\bx_s)=\alpha G(\bx_1,\ldots,\bx_s)+R_{\alpha,\omega}(\bx_1,\ldots,\bx_s).
\end{equation}
Define
\begin{equation}\label{eq:T-uniform}
 T_\omega(\alpha)=\sum_{\bx_1\in B_1}\cdots\sum_{\bx_s\in B_s}\psi\bigl(\Phi_{\alpha,\omega}(\bx_1,\ldots,\bx_s)\bigr).
\end{equation}

For each $j$, introduce $m_j$ independent copies
\[
 \bz_{j,1},\ldots,\bz_{j,m_j}\in B_j.
\]
Put
\[c_*=c\prod_{j=1}^s m_j!.\]
 The assumption $p>d$ implies $c_*\neq 0$ in $\F_q$. Define the completely multilinear sum
\begin{align}
 E(\alpha)
 &=\sum_{\substack{\bz_{j,\ell}\in B_j\\ 1\leq j\leq s,\ 1\leq\ell\leq m_j}} \psi\Bigl(\alpha c_*\, \Gamma_f(\bz_{1,1},\ldots,\bz_{1,m_1},\ldots, \bz_{s,1},\ldots,\bz_{s,m_s})\Bigr).
 \label{eq:E}
\end{align}
In particular,
\[
E(0)=\prod_{j=1}^s M_j^{m_j}.
\]

We prove the following uniform multiblock Weyl differencing.

\begin{proposition}[Multiblock Weyl differencing]\label{prop:uniform}
Under the assumptions above, for every $\alpha\in\mathbb T$ and every
$\omega\in\Omega$ one has
\[
|T_\omega(\alpha)|^{2^{d-1}} \leq M^{2^{d-1}}\frac{E(\alpha)}{E(0)}.
\]
Moreover, $E(\alpha)$ is a nonnegative integer. The above estimate is uniform in $\omega$ and in all coefficients of $R_{\alpha,\omega}$.
\end{proposition}

The rest of this section is devoted to the proof of the above proposition. For $1\leq j\leq s$ and $\bh\in B_j$, define the difference operator in block $j$ by

\begin{align}
 \Delta_{\bh}^{(j)}H(\bx_1,\ldots,\bx_s) &=H(\bx_1,\ldots,\bx_j+\bh,\ldots,\bx_s)-H(\bx_1,\ldots,\bx_j,\ldots,\bx_s).
 \label{eq:difference}
\end{align}
Note that each $B_j$ is an additive group. We first prove a few preliminary lemmas, starting with basic Weyl differencing. For a finite nonempty set $A$ and a function $F:A\to \mathbb C$, we use the normalized average notation
\[
\mathbb E_{x\in A} F(x):=\frac{1}{\#A}\sum_{x\in A}F(x).
\]
For several variables, $\mathbb E_{x_1,\ldots,x_s}$ denotes the corresponding normalized average over the product of their respective ranges.

\begin{Lemma}\label{lem:one-step}
Let $H:B_1\times\cdots\times B_s\to \mathbb K_\infty$ be any function, and fix a block index $j$. Then
\[
 \left|\E_{\bx_1,\ldots,\bx_s}\psi(H(\bx_1,\ldots,\bx_s))\right|^2 \leq \E_{\bh\in B_j} \left| \E_{\bx_1,\ldots,\bx_s} \psi\bigl(\Delta_{\bh}^{(j)}H(\bx_1,\ldots,\bx_s)\bigr) \right|.
\]
\end{Lemma}

\begin{proof}
Let $\bx'$ denote all blocks other than $\bx_j$. By Cauchy--Schwarz in the $\bx'$-average,
\begin{align}
 \left|\E_{\bx'}\E_{\bx_j}\psi(H(\bx_j,\bx'))\right|^2 &\leq \E_{\bx'} \left|\E_{\bx_j}\psi(H(\bx_j,\bx'))\right|^2.
 \label{eq:CS-outer}
\end{align}
For fixed $\bx'$, expanding we get
\[
 \left|\E_{\bx_j}\psi(H(\bx_j,\bx'))\right|^2 =\frac{1}{M_j^2} \sum_{\bx_j,\by_j\in B_j} \psi\bigl(H(\bx_j,\bx')-H(\by_j,\bx')\bigr).
\]
Note that the change of variables
\[
 (\bx_j,\by_j)=(\by_j+\bh,\by_j)
\]
is a bijection from $B_j\times B_j$ to itself. Hence
\[
 \left|\E_{\bx_j}\psi(H(\bx_j,\bx'))\right|^2 =\E_{\bh\in B_j}\E_{\by_j\in B_j} \psi\bigl(\Delta_{\bh}^{(j)}H(\by_j,\bx')\bigr).
\]
Substituting this into \eqref{eq:CS-outer} we obtain
\[
\left|\E_{\bx}\psi(H(\bx))\right|^2\leq \E_{\bh\in B_j} \E_{\bx} \psi\bigl(\Delta_{\bh}^{(j)}H(\bx)\bigr).
\]
The right side as a whole is real and nonnegative, but the inner average for a fixed $\bh$ need not be. Applying the triangle inequality to the $\bh$-average proves the lemma.
\end{proof}
Iterating the preceding lemma a number of times to any prescribed sequence of blocks gives the following form.

\begin{Lemma}
\label{lem:iterated}
Let $i_1,\ldots,i_k\in\{1,\ldots,s\}$ be any sequence of block indices. For $1\leq r\leq k$, let $\bh_r\in B_{i_r}$ and put
\[
 \mathcal D_r=\Delta_{\bh_r}^{(i_r)}\cdots\Delta_{\bh_1}^{(i_1)}.
\]
Then, for any function $H:B_1\times\cdots\times B_s\to \mathbb K_\infty$,
\[
 \left|\E_{\bx}\psi(H(\bx))\right|^{2^k}\leq \E_{\bh_1\in B_{i_1}}\cdots \E_{\bh_k\in B_{i_k}} \left|\E_{\bx}\psi(\mathcal D_kH(\bx))\right|.
\]
\end{Lemma}

\begin{proof}
We induct on $k$. The case $k=1$ is Lemma \ref{lem:one-step}. Suppose the result is known for $k$. Write
\[
 A_k(\bh_1,\ldots,\bh_k) =\E_{\bx}\psi(\mathcal D_kH(\bx)).
\]
By the induction hypothesis,
\[
 \left|\E_{\bx}\psi(H(\bx))\right|^{2^k} \leq \E_{\bh_1,\ldots,\bh_k}|A_k(\bh_1,\ldots,\bh_k)|.
\]
Squaring both sides, since the increment averages are probability averages, Cauchy--Schwarz gives
\begin{align}
 \left|\E_{\bx}\psi(H(\bx))\right|^{2^{k+1}}
 &\leq
 \E_{\bh_1,\ldots,\bh_k}|A_k(\bh_1,\ldots,\bh_k)|^2.
 \label{eq:induction-CS}
\end{align}
For each fixed tuple $(\bh_1,\ldots,\bh_k)$, apply Lemma \ref{lem:one-step} to the phase $\mathcal D_kH$ and to block $i_{k+1}$. This gives
\begin{align*}
 |A_k(\bh_1,\ldots,\bh_k)|^2 &\leq \E_{\bh_{k+1}\in B_{i_{k+1}}}\left|\E_{\bx}\psi\bigl(\Delta_{\bh_{k+1}}^{(i_{k+1})}\mathcal D_kH(\bx)\bigr)\right|.
\end{align*}
Substituting into \eqref{eq:induction-CS} proves the statement for $k+1$, completing the inductive step.
\end{proof}

Isolating monomials we note that for each $j$, $\Delta_\bh^{(j)}$ is a degree reducing operator in $\bx_j$. Next, we prove an elementary lemma on complete and almost complete differences.

\begin{Lemma}\label{lem:multilinear-difference}
Let $V$ be a vector space over a field, let $W$ be another vector space, and let $\Lambda:V^m\to W$ be symmetric and $m$-linear. Set
\[
 Q(\bx)=\Lambda(\bx,\ldots,\bx).
\]
Then
\begin{align}
 \Delta_{\bh_m}\cdots\Delta_{\bh_1}Q(\bx) &=m!\,\Lambda(\bh_1,\ldots,\bh_m),
 \label{eq:complete-diff}
\end{align}
and
\begin{align}
 \Delta_{\bh_{m-1}}\cdots\Delta_{\bh_1}Q(\bx) &=m!\,\Lambda(\bh_1,\ldots,\bh_{m-1},\bx) +C(\bh_1,\ldots,\bh_{m-1}),
 \label{eq:almost-complete-diff}
\end{align}
where $C$ is independent of $\bx$.
\end{Lemma}

\begin{proof}
The inclusion-exclusion formula for iterated differences is
\begin{align}
 \Delta_{\bh_k}\cdots\Delta_{\bh_1}Q(\bx) &=\sum_{I\subseteq\{1,\ldots,k\}} (-1)^{k-|I|} Q\left(\bx+\sum_{i\in I}\bh_i\right).
 \label{eq:inc-exc}
\end{align}
We expand each value of $Q$ by multilinearity. Consider a term in the expansion and fix an increment $\bh_r$. If the term does not contain $\bh_r$, then its total coefficient in \eqref{eq:inc-exc} is zero, because the contributions from subsets $I$ containing $r$ cancel those from subsets not containing $r$. Hence a term can survive only if it contains every increment $\bh_1,\ldots,\bh_k$ at least once.

If $k=m$, the total number of multilinear slots is $m$, so every surviving term contains each of the $m$ increments exactly once and contains no $\bx$. The symmetry of $\Lambda$ shows that all such terms are equal, and there are $m!$ permutations of the increments among the slots. This proves \eqref{eq:complete-diff}.

If $k=m-1$, every surviving term that still depends on $\bx$ contains each of the $m-1$ increments exactly once and contains $\bx$ in the remaining slot. There are $m!$ ways to place these $m$ entries among the slots, and all resulting values are equal by symmetry. This gives the first term on the right of \eqref{eq:almost-complete-diff}. Every other surviving term contains no $\bx$ and can therefore be absorbed into the function $C$.
\end{proof}

We are now ready to prove Proposition \ref{prop:uniform}.

\begin{proof}[Proof of Proposition \ref{prop:uniform}]
 We have that
 \[
 \E_{\bx_1,\ldots,\bx_s} \psi(\Phi_{\alpha,\omega}(\bx_1,\ldots,\bx_s)) =\frac{T_\omega(\alpha)}{M}.
 \]
 We distinguish the last block corresponding to $s$. For blocks with $j<s$, we choose exactly $m_j$ increments
\[
 \bh_{j,1},\ldots,\bh_{j,m_j}\in B_j.
\]
In block $s$, choose exactly $m_s-1$ increments
\[
 \bh_{s,1},\ldots,\bh_{s,m_s-1}\in B_s.
\]
Here we use the convention that this last list is empty if $m_s=1$. The total number of increments is
\[
 \sum_{j=1}^{s-1}m_j+(m_s-1)=d-1.
\]
We consider the combined difference operator
\begin{align}
 \mathcal D &=\left(\prod_{\ell=1}^{m_s-1} \Delta_{\bh_{s,\ell}}^{(s)}\right) \prod_{j=1}^{s-1} \left(\prod_{\ell=1}^{m_j}\Delta_{\bh_{j,\ell}}^{(j)}\right).
 \label{eq:D}
\end{align}
Note that the difference operators commute, so their written order is irrelevant.

The number of possible increment tuples is
\begin{equation}\label{eq:D-size}
 D_0=\left(\prod_{j=1}^{s-1}M_j^{m_j}\right)M_s^{m_s-1}.
\end{equation}
Thus the normalised increment average is
\[
 \E_{\bh}[-]=\frac{1}{D_0}\sum_{\substack{\bh_{j,\ell}\in B_j\\ j<s,\ 1\leq\ell\leq m_j\\ j=s,\ 1\leq\ell\leq m_s-1}}-.
\]

Applying Lemma \ref{lem:iterated} to the phase $\Phi_{\alpha,\omega}$ and to this sequence of $d-1$ block differences. We obtain
\begin{align}
 \left|\frac{T_\omega(\alpha)}{M}\right|^{2^{d-1}} &\leq \E_{\bh} \left| \E_{\bx_1,\ldots,\bx_s} \psi\bigl(\mathcal D\Phi_{\alpha,\omega}(\bx_1,\ldots,\bx_s)\bigr) \right|.
 \label{eq:after-weyl}
\end{align}

Next we compute the differenced leading term $\mathcal DG$. In every block $j<s$, the form $G$ has degree exactly $m_j$, and we apply exactly $m_j$ differences. In block $s$, it has degree exactly $m_s$, and we apply $m_s-1$ differences. Applying Lemma \ref{lem:multilinear-difference} successively in the different blocks gives
\begin{align}
 \mathcal DG(\bx_1,\ldots,\bx_s) &=c_*\,\Gamma_f(\bh_{1,1},\ldots,\bh_{1,m_1},\ldots, \bh_{s,1},\ldots,\bh_{s,m_s-1},\bx_s)+C_G(\bh),
 \label{eq:DG}
\end{align}
where $C_G$ is independent of every original variable and $c_*$ is given by \eqref{eq:E}.

We now estimate the differenced remainder. Fix a monomial $\mathfrak m$ occurring in $R_{\alpha,\omega}$. By the deficiency hypothesis, there exists $j$ such that
\[
 \deg_{\bx_j}\mathfrak m<m_j.
\]
If this happens for some $j<s$, then the $m_j$ differences taken in block $j$ annihilate the monomial. Therefore
\[
 \mathcal D\mathfrak m=0.
\]
Otherwise, the monomial is not deficient in any block $j<s$. Its required deficiency must then occur in block $s$, so
\[
 \deg_{\bx_s}\mathfrak m<m_s.
\]
After the $m_s-1$ differences in block $s$, the resulting degree in $\bx_s$ is at most
\[
 \deg_{\bx_s}\mathfrak m-(m_s-1)\leq 0.
\]
Thus the differenced monomial is independent of $\bx_s$. It may still depend on $\bx_1,\ldots,\bx_{s-1}$ and on all increments; that dependence is harmless.

Summing over the monomials of $R_{\alpha,\omega}$, we conclude that
\begin{equation}\label{eq:DR}
 \mathcal DR_{\alpha,\omega}= C_{R,\alpha,\omega}(\bx_1,\ldots,\bx_{s-1};\bh),
\end{equation}
where the right side is independent of $\bx_s$. Combining the leading and remainder terms we deduce that
\begin{align}
 \mathcal D\Phi_{\alpha,\omega}(\bx_1,\ldots,\bx_s) &=\alpha c_*\,\Gamma_f(\bh_{1,1},\ldots,\bh_{s,m_s-1},\bx_s) +C_{\alpha,\omega}(\bx_1,\ldots,\bx_{s-1};\bh),
 \label{eq:Dphi}
\end{align}
where $C_{\alpha,\omega}$ is independent of $\bx_s$.

Note that fixing an increment $\bh$ we get
\begin{align}
 &\E_{\bx_1,\ldots,\bx_s} \psi\bigl(\mathcal D\Phi_{\alpha,\omega}(\bx_1,\ldots,\bx_s)\bigr) \notag\\ &\quad= \E_{\bx_1,\ldots,\bx_{s-1}} \left[ \psi\bigl(C_{\alpha,\omega}(\bx_1,\ldots,\bx_{s-1};\bh)\bigr) \E_{\bx_s}\psi\bigl(\alpha c_*\,\Gamma_f(
 \bh_{1,1},\ldots,\bh_{s,m_s-1},\bx_s)\bigr) \right].
 \label{eq:factor-inner}
\end{align}
The $\bx_s$-average is independent of $\bx_1,\ldots,\bx_{s-1}$. Taking absolute values and using $|\psi(C_{\alpha,\omega})|=1$, we get
\begin{align}
 \left| \E_{\bx_1,\ldots,\bx_s} \psi\bigl(\mathcal D\Phi_{\alpha,\omega}\bigr) \right|\leq \left|\E_{\bx_s\in B_s} \psi\bigl(\alpha c_*\,\Gamma_f( \bh_{1,1},\ldots,\bh_{s,m_s-1},\bx_s)\bigr)\right|.
 \label{eq:discard-remainder}
\end{align}

Using linearity and orthogonality \eqref{sum-orth-id} we deduce that
\begin{align}
 \E_{\bx_s\in B_s} \psi\bigl(\alpha c_*\,\Gamma_f( \bh_{1,1},\ldots,\bh_{s,m_s-1},\bx_s)\bigr) &=I_\alpha(\bh),
 \label{eq:I}
\end{align}
where $I_\alpha(\bh)$ is the indicator of the $n$ simultaneous conditions
\begin{equation}\label{eq:dual-conditions}
 \left\|\alpha c_*\Gamma_f(\bh_{1,1},\ldots,\bh_{s,m_s-1},\be_i)\right\|<q^{-P_s} \qquad (1\leq i\leq n).
\end{equation}
Therefore, we obtain
\begin{equation}\label{eq:ratio-N}
 \left|\frac{T_\omega(\alpha)}{M}\right|^{2^{d-1}} \leq \E_{\bh}I_\alpha(\bh) =\frac{N(\alpha)}{D_0},
\end{equation}
where $N(\alpha)$ is the number of increment tuples satisfying \eqref{eq:dual-conditions}. We now identify $E(\alpha)$ in the above bound. Regarding the last variable $\bz_{s,m_s}$ as the last summation variable and fixing all the other $d-1$ variables. By definition \eqref{eq:E}, the inner sum is
\[
 \sum_{\bz_{s,m_s}\in B_s} \psi\Bigl(\alpha c_*\Gamma_f(\bz_{1,1},\ldots,\bz_{s,m_s-1},\bz_{s,m_s})\Bigr).
\]
By orthogonality again, this sum is $M_s$ if the fixed $d-1$ variables satisfy the conditions \eqref{eq:dual-conditions}, and is zero otherwise. Consequently,
\begin{equation}\label{eq:E=N}
 E(\alpha)=M_sN(\alpha).
\end{equation}
When $\alpha=0$, every increment tuple satisfies the dual conditions, so $N(0)=D_0$. Combining \eqref{eq:ratio-N} and \eqref{eq:E=N} we get
\[
|T_\omega(\alpha)|^{2^{d-1}} \leq M^{2^{d-1}}\frac{E(\alpha)}{E(0)}.
\]
This completes the proof.
\end{proof}

\section{A mean value estimate}\label{sec:mean-value}

This section is devoted to estimating the multilinear sum $E(\alpha)$. We deal with a slightly more general sum, but we stick with our previous notation. Let $P_1,\ldots,P_d\geq 1$ be arbitrary integers, put
\begin{equation}\label{eq:H-def}
 H=P_1+\cdots+P_d, \qquad P_{\min}=\min_i P_i,
\end{equation}
and let
\begin{equation}\label{eq:E-general}
 \tilde E(\alpha)= \sum_{\mathbf z_1\in\mathcal B(P_1)}\cdots \sum_{\mathbf z_d\in\mathcal B(P_d)} \psi\bigl(\alpha\Lambda(\mathbf z_1,\ldots,\mathbf z_d)\bigr),
\end{equation}
where $\Lambda=c\Gamma_f$ for some $c\in\F_q^\times$.
 Define the affine variety
\[
 Z=\left\{(\mathbf z_1,\ldots,\mathbf z_{d-1})\in\mathbb A^{(d-1)n}: \Lambda(\mathbf z_1,\ldots,\mathbf z_{d-1},\mathbf e_i)=0 \ \text{for all }i\right\},
\]
and set
\[
 \sigma=(d-1)n-\dim Z.
\]

\begin{Lemma}\label{lem:sigma}
If $f$ is nonsingular, then $\sigma\geq n$.
\end{Lemma}

\begin{proof}
Let
\[
 D=\{(\mathbf x,\ldots,\mathbf x):\mathbf x\in\mathbb A^n\} \subset\mathbb A^{(d-1)n}.
\]
By \eqref{eq:gradient-polarisation}, the intersection $Z\cap D$ consists of vectors $\mathbf x$ satisfying $\nabla f(\mathbf x)=0$. Since $f$ is a nonsingular homogeneous form and $p>d$, only such affine vector is $\mathbf 0$. Thus $\dim(Z\cap D)=0$. Since $Z$ is a cone every irreducible component meets $D$ at origin. The affine dimension theorem gives the desired bound $\sigma\geq n$.
\end{proof}

The next lemma proves a major and minor arc dichotomy for the sum under consideration.

\begin{Lemma}[Major and Minor arcs]\label{lem:maj-min}
Let $1\leq J\leq P_{\min}$. Then one of the following alternatives holds:
\begin{enumerate}
\item the following bound
\[
 |\tilde E(\alpha)|\leq (d-1)^n \tilde E(0)q^{-\sigma J};
\]
\item there exist coprime $a,g\in\F_q[u]$, with $g$ monic, such that
\[
 0<|g|\leq q^{(d-1)(J-1)},
 \qquad
 |g\alpha-a|<q^{-H+(d-1)J}.
\]
\end{enumerate}
\end{Lemma}

\begin{proof}
Reorder the variables so that $P_d=\max_i P_i$; this is harmless because $\Lambda$ is symmetric. We leave the last variable $\mathbf z_d$ undifferenced. Exactly as in the proof of Proposition \ref{prop:uniform}, using orthogonality,
\[
 \tilde E(\alpha)=q^{nP_d}\tilde N(\alpha),
\]
where $\tilde N(\alpha)$ counts
$(\mathbf z_1,\ldots,\mathbf z_{d-1})$ satisfying
\[
 |\mathbf z_i|<q^{P_i} \quad(1\leq i\leq d-1),
\]
and
\[
 \|\alpha\Lambda(\mathbf z_1,\ldots,\mathbf z_{d-1},\mathbf e_j)\| <q^{-P_d} \quad(1\leq j\leq n).
\]

Define
\[
B_0=P_d
\]
and, for $1\leq t\leq d-1$,
\[
B_t = P_d+\sum_{i=1}^{t}(P_i-J).
\]
Thus
\[
B_t-B_{t-1}=P_t-J.
\]

For $0\leq t\leq d-1$, let $\tilde N^{(t)}(\alpha)$ denote the number of tuples
\[
(\mathbf z_1,\ldots,\mathbf z_{d-1})
\]
satisfying
\[
|{\mathbf z_i}|<q^J,
\qquad 1\leq i\leq t,
\]
\[
|{\mathbf z_i}|<q^{P_i},
\qquad t<i\leq d-1,
\]
and
\[
\| \alpha\Lambda(\mathbf z_1,\ldots,\mathbf z_{d-1},\mathbf e_j) \|< q^{-B_t}, \qquad 1\leq j\leq n.
\]
We claim that for every $1\leq t\leq d-1$,
\[
\tilde N^{(t-1)}(\alpha)\leq q^{n(P_t-J)}\tilde N^{(t)}(\alpha).
\]
To see this, fix $t$, and freeze all variables except $\mathbf z_t$. Write
\[
\mathbf z_t=(z_{t,1},\ldots,z_{t,n}).
\]
For each $1\leq j\leq n$, define a linear form in the coordinates of $\mathbf z_t$ by
\[
L_j(\mathbf z_t) = \alpha \Lambda( \bz_1,\ldots,\bz_{t-1}, \bz_t, \bz_{t+1},\ldots,\bz_{d-1}, \be_j
).
\]
Since $\Lambda$ is linear in the $t$-th slot, we may write
\[
L_j(\mathbf z_t) = \sum_{k=1}^n \gamma_{jk}z_{t,k},
\]
where
\[
\gamma_{jk} = \alpha \Lambda( \bz_1,\ldots,\bz_{t-1}, \be_k, \bz_{t+1},\ldots,\bz_{d-1}, \mathbf e_j ).
\]
Note that $L_1,\ldots,L_n$ form a symmetric system of $n$ linear forms in $n$ variables.

For the fiber occurring in $\tilde N^{(t-1)}(\alpha)$, the variable $\mathbf z_t$ satisfies
\[
|{\bz_t}|<q^{P_t}
\]
and
\[
\|L_j(\bz_t)\|<q^{-B_{t-1}}, \qquad 1\leq j\leq n.
\]
Using the notation in Lemma \ref{lem:shrinking}, we set $A_t=\frac{P_t+B_{t-1}}{2}, Z_2=\frac{P_t-B_{t-1}}{2}$. Then $A_t+Z_2=P_t$ and $-A_t+Z_2=-B_{t-1}$. Therefore the original fiber is exactly $\mathcal N_{A_t}(Z_2)$. For the shrunken fiber, set $Z_1=J-A_t$, then $A_t+Z_1=J$ while $-A_t+Z_1= -B_t$. Hence $\mathcal N_{A_t}(Z_1)$ counts precisely those
$\mathbf z_t$ satisfying
\[
|{\mathbf z_t}|<q^J
\]
and
\[
\|{L_j(\mathbf z_t)}\|<q^{-B_t},
\qquad 1\leq j\leq n.
\]
Note that
\[
 B_{t-1}\geq P_d\geq P_t,\qquad J\leq P_t.
\]
Consequently $Z_2\leq0$, $Z_1\leq Z_2$, and
$A_t-Z_2=B_{t-1}\geq0$; moreover $A_t\pm Z_i$ are integers. Shrinking, Lemma \ref{lem:shrinking}, therefore gives
\[
\mathcal N_{A_t}(Z_2) \leq q^{n(Z_2-Z_1)} \mathcal N_{A_t}(Z_1) = q^{n(P_t-J)} \mathcal N_{A_t}(Z_1).
\]
This estimate holds for every fixed choice of the remaining variables. Summing it over all allowed values of those frozen variables yields
\[
\tilde N^{(t-1)}(\alpha)
\leq
q^{n(P_t-J)}\tilde N^{(t)}(\alpha).
\]
The above application of Lemma \ref{lem:shrinking} yields
\begin{equation}\label{eq:multi-shrink}
 \tilde N(\alpha)
 \leq q^{n\sum_{i=1}^{d-1}(P_i-J)}\tilde N_J(\alpha),
\end{equation}
where $\tilde N_J(\alpha)$ counts tuples with
\[
 |\mathbf z_i|<q^J\quad(1\leq i\leq d-1)
\]
and
\begin{equation}\label{eq:NJ-dual}
 \|\alpha\Lambda(\mathbf z_1,\ldots,\mathbf z_{d-1},\mathbf e_j)\| <q^{-H+(d-1)J} \quad(1\leq j\leq n).
\end{equation}
Suppose first that every tuple counted by $\tilde N_J(\alpha)$ belongs to $Z$. The variety $Z$ is cut out by $n$ equations of degree $d-1$, so $\delta(Z)\leq(d-1)^n$. By Lemma \ref{lem:poly-points},
\[
 \tilde N_J(\alpha)\leq(d-1)^n q^{J\dim Z}.
\]
Combining this with the bound for $\tilde E(\alpha)$ and \eqref{eq:multi-shrink}, and using
$\tilde E(0)=q^{nH}$, gives
\[
 |\tilde E(\alpha)| \leq(d-1)^n q^{np_d+n\sum_{i<d}(P_i-J)+J\dim Z} =(d-1)^n \tilde E(0)q^{-\sigma J}.
\]
This is the first alternative.

Otherwise, choose a tuple counted by $\tilde N_J(\alpha)$ that is not in $Z$. For some $j$ the polynomial
\[
 g_0=\Lambda(\bz_1,\ldots,\bz_{d-1},\be_j)
\]
is nonzero. Since every $\bz_i$ has degree strictly less than $J$,
\[
 0<|g_0|\leq q^{(d-1)(J-1)}.
\]
By \eqref{eq:NJ-dual}, there is $a_0\in\F_q[u]$ such that
\[
 |g_0\alpha-a_0|<q^{-H+(d-1)J}.
\]
Divide $g_0$ and $a_0$ by their greatest common divisor and make the new
denominator monic. This proves the second alternative.
\end{proof}

For $J\geq 1$, define
\begin{equation}\label{eq:major-arcs}
 \mathfrak M(J)= \bigcup_{\substack{g\in\F_q[u]\ \mathrm{monic}\\0<|g|\leq q^{(d-1)(J-1)}}}
 \ \bigcup_{\substack{a\in\F_q[u]\\|a|<|g|, (a,g)=1}}
 \left\{\alpha\in\mathbb T:
 \left|\alpha-\frac ag\right|
 <\frac{q^{-H+(d-1)J}}{|g|}
 \right\}.
\end{equation}
By Lemma \ref{lem:dirichlet},
\[
\mathfrak M(J)=\mathbb T \quad\text{whenever}\quad 2(d-1)J\geq H+d-1.
\]

\begin{proposition}[Mean value estimate]\label{prop:mean-value}
Let $\rho>0$. Suppose
\[
\sigma\rho>2(d-1)
\]
and
\[
 P_{\min}\geq \frac{H-d+1}{2(d-1)}.
\]
Put $\delta=\sigma\rho-2(d-1)>0$. Then
\begin{equation}\label{eq:mean-value}
 \int_\mathbb T |\tilde E(\alpha)|^\rho\,d\alpha \leq \tilde E(0)^\rho q^{-H+d-1} \left(1+(d-1)^{n\rho}\frac{q^{-\delta}}{1-q^{-\delta}}\right).
\end{equation}
\end{proposition}

\begin{proof}
 We bound contribution from $\mathfrak M(1)$ trivially. Only possible denominator in \eqref{eq:major-arcs} for $J=1$ has degree zero, using the normalization
\[
 d\alpha(\mathfrak M(1))\leq q^{-H+d-1},
\]
and therefore using triangle inequality
\begin{equation}\label{eq:M1}
 \int_{\mathfrak M(1)}|\tilde E(\alpha)|^\rho\,d\alpha \leq \tilde E(0)^\rho q^{-H+d-1}.
\end{equation}
For $J\geq 1$, put
\[
 I_J=\int_{\mathfrak M(J+1)\setminus\mathfrak M(J)}|\tilde E(\alpha)|^\rho\,d\alpha.
\]
If $J> P_{\min}$, then
\[
 2(d-1)J\geq 2(d-1)(P_{\min}+1)\geq H+d-1.
\]
Thus $\mathfrak M(J)=\mathbb T$ by Lemma \ref{lem:dirichlet}, and hence $I_J=0$. Now let $1\leq J\leq P_{\min}$. On $\mathfrak M(J+1)\setminus\mathfrak M(J)$, the second alternative of Lemma \ref{lem:maj-min} with parameter $J$ is impossible, so
\[
 |\tilde E(\alpha)|^\rho \leq(d-1)^{n\rho}\tilde E(0)^\rho q^{-\sigma\rho J}.
\]
For a fixed denominator $g$, summing the measures of the arcs over all reduced numerators $a$ gives the bound
\[
d\alpha(\mathfrak M(J+1)) \leq \sum_{\substack{\deg g\leq (d-1)J,\\ \mathrm{monic}}} \phi(g)\cdot \frac{q^{-H+(d-1)(J+1)}}{|g|},
\]
where $\phi$ is the usual Euler totient function. Using the fact (see, e.g. \cite{Rosen2002})
\[
\sum_{\substack{\deg g\leq (d-1)J,\\ \mathrm{monic}}} \frac{\phi(g)}{|g|}= q^{(d-1)J},
\]
we obtain
\[
d\alpha(\mathfrak M(J+1)) \leq q^{-H+(d-1)(J+1)+(d-1)J}.
\]
It follows that
\begin{align*}
 I_J &\leq(d-1)^{n\rho}\tilde E(0)^\rho q^{-H+(d-1)(J+1)+(d-1)J-\sigma\rho J}\\
 &= (d-1)^{n\rho}\tilde E(0)^\rho q^{-H+d-1-\delta J}.
\end{align*}
Summing this estimate over $J\geq1$ and adding \eqref{eq:M1} proves \eqref{eq:mean-value}.
\end{proof}

\section{Finite field counting}\label{sec:finite-field-count}

In this section we count all coefficient tuples satisfying $f(g_1,\ldots,g_n)=0$, without imposing the base-point-free condition. Thus the counting problem concerns the affine cone over $Q_e(X)$; the morphism locus will be recovered later as a nonempty open subset.

\subsection{The coefficient system} For $m\geq0$, put
\[
 \Delta_{r-1}(m)= \{\ba=(a_1,\ldots,a_{r-1})\in\mathbb Z_{\geq0}^{r-1}: |\ba|\leq m\}.
\]
Thus
\[
 \#\Delta_{r-1}(m)=\binom{m+r-1}{r-1}.
\]

We work on the standard affine chart of $\mathbb P^r$ with coordinates $u,z_1,\ldots,z_{r-1}$. Every polynomial of total degree at most $e$ has a unique expansion
\begin{equation}\label{eq:g-expansion}
 g_i(u,\bz)= \sum_{\ba\in\Delta_{r-1}(e)} t_{i,\ba}(u)\mathbf z^{\ba}, \qquad \deg t_{i,\ba}\leq e-|\ba|.
\end{equation}
Set
\[
 \mathbf t_{\ba}=(t_{1,\mathbf a},\ldots,t_{n,\ba}) \in\F_q[u]^n, \qquad h(\ba)=e-|\ba|+1.
\]
Thus $|\mathbf t_{\mathbf a}|<q^{h(\ba)}$.

Expanding by polarisation gives
\begin{equation}\label{eq:f-expansion}
 f(g_1,\ldots,g_n)= \sum_{\boldsymbol\beta\in\Delta_{r-1}(de)} F_{\boldsymbol\beta}(\mathbf t)\mathbf z^{\boldsymbol\beta},
\end{equation}
where
\begin{equation}\label{eq:F-beta}
 F_{\boldsymbol\beta}(\mathbf t)= \sum_{\substack{\ba_1,\ldots,\ba_d\in\Delta_{r-1}(e)\\ \ba_1+\cdots+\ba_d=\boldsymbol\beta}}
 \Gamma_f(\mathbf t_{\ba_1},\ldots,\mathbf t_{\ba_d}).
\end{equation}
The polynomial $F_{\boldsymbol\beta}$ has degree at most $de-|\boldsymbol\beta|$ in $u$. Define
\begin{equation}\label{eq:h-beta}
 h_d(\boldsymbol\beta)=de-|\boldsymbol\beta|+1.
\end{equation}
Let
\begin{equation}\label{eq:U-def}
 \mathcal U= \prod_{\mathbf a\in\Delta_{r-1}(e)}\mathcal B(h(\mathbf a)).
\end{equation}
This section is devoted to estimating the following quantity over finite fields.
\begin{equation}\label{eq:N-def}
 N_{r, q}(e)=\#\left\{\mathbf t\in\mathcal U: F_{\boldsymbol\beta}(\mathbf t)=0\ \text{for all }\boldsymbol\beta\in\Delta_{r-1}(de)\right\}.
\end{equation}

We record the following elementary lemma for convenience.

\begin{Lemma}\label{lem:lattice-identities}
For every $m\geq0$,
\[
 \sum_{\mathbf a\in\Delta_{r-1}(m)}(m-|\mathbf a|+1) =\binom{m+r}{r}.
\]
Consequently,
\[
 \#\mathcal U=q^{n\binom{e+r}{r}}, \qquad \sum_{\boldsymbol\beta\in\Delta_{r-1}(de)} h_d(\boldsymbol\beta)=\binom{de+r}{r}.
\]
\end{Lemma}

\begin{proof}
The left hand side counts pairs $(\mathbf a,j)$ with $\mathbf a\in\mathbb Z_{\geq0}^{r-1}$ and $0\leq j\leq m-|\mathbf a|$, these are exactly the points of $\Delta_r(m)$, whose cardinality is $\binom{m+r}{r}$.
\end{proof}

\subsection{Coefficient isolation} In this section we isolate coefficients in the circle method phase. Let us begin by decomposing $\boldsymbol\beta$ into a balanced form so we can apply our multiblock Weyl differencing and mean value estimate.

\begin{Lemma}\label{lem:balanced}
For every $\boldsymbol\beta\in\Delta_{r-1}(de)$, there are $\ba_1,\ldots,\ba_d\in\Delta_{r-1}(e)$ such that
\[
 \ba_1+\cdots+\ba_d=\boldsymbol\beta
\]
and the integers $|\ba_i|$ differ by at most one.
\end{Lemma}

\begin{proof}
Let us write
\[
 |\boldsymbol\beta|=dw+s, \qquad 0\leq s<d.
\]
Regarding $\boldsymbol\beta$ as a multiset containing $\beta_j$ copies of the $\be_j$, standard basis vector of $\mathbb Z^{r-1}$, we partition this multiset into $d$ piles, $s$ of cardinality $w+1$ and $d-s$ of cardinality $w$. Let $\ba_i$ be the sum of the vectors in the $i$th pile. Then the required sum identity holds and $|\mathbf a_i|\in\{w,w+1\}$. Since $|\boldsymbol\beta|\leq de$, one has $w\leq e$, and if $w=e$ then $s=0$. Hence every $\ba_i$ belongs to $\Delta_{r-1}(e)$.
\end{proof}

We fix such a decomposition for every $\boldsymbol\beta$. Let $\boldsymbol\gamma_1,\ldots,\boldsymbol\gamma_s$ be distinct vectors among $\ba_1,\ldots,\ba_d$, with, say, multiplicities $m_1,\ldots,m_s$. Then
\begin{equation}\label{eq:multiplicity-identities}
 m_1+\cdots+m_s=d,\qquad m_1\boldsymbol\gamma_1+\cdots+m_s\boldsymbol\gamma_s =\boldsymbol\beta.
\end{equation}
Putting
\begin{equation}\label{eq:P-beta}
 P_j=h(\boldsymbol\gamma_j)=e-|\boldsymbol\gamma_j|+1, \qquad H_{\boldsymbol\beta}=\sum_{j=1}^s m_jP_j,
\end{equation}
we deduce that
\begin{equation}\label{eq:H-beta}
 H_{\boldsymbol\beta}=d(e+1)-|\boldsymbol\beta|,\qquad H_{\boldsymbol\beta}-d+1=h_d(\boldsymbol\beta).
\end{equation}
We immediately get the following.

\begin{Lemma}\label{lem:height-threshold}
For every $\boldsymbol\beta$, the balanced decomposition satisfies
\[
 \min_jP_j\geq \frac{H_{\boldsymbol\beta}-d+1}{2(d-1)}.
\]
\end{Lemma}
\begin{proof}
    The $d$ side lengths $P_j$, counted with multiplicity $m_j$, differ by at most one,
\[
H_\beta\le P_{\min}+(d-1)(P_{\min}+1)
=dP_{\min}+d-1\le 2(d-1)P_{\min}+d-1,
\]
which gives the result.
\end{proof}

We now state the coefficient isolating lemma.

\begin{Lemma}\label{lem:isolation}
Let
\[
 \Phi_{\boldsymbol\alpha}(\mathbf t)= \sum_{\boldsymbol\eta\in\Delta_{r-1}(de)} \alpha_{\boldsymbol\eta}F_{\boldsymbol\eta}(\mathbf t),
 \qquad \boldsymbol\alpha\in\mathbb T^{\Delta_{r-1}(de)}.
\]
Fix $\boldsymbol\beta$, and regard this phase as a polynomial in the selected blocks $\mathbf t_{\boldsymbol\gamma_1},\ldots,
\mathbf t_{\boldsymbol\gamma_s}$, keeping every other block fixed. Then
\begin{equation}\label{eq:isolation}
 \Phi_{\boldsymbol\alpha}(\mathbf t)= \alpha_{\boldsymbol\beta}G_{\boldsymbol\beta} (\mathbf t_{\boldsymbol\gamma_1},\ldots, \mathbf t_{\boldsymbol\gamma_s})+R_{\boldsymbol \alpha, \boldsymbol\beta}(\mathbf t),
\end{equation}
where
\begin{equation}\label{eq:G-beta}
 G_{\boldsymbol\beta}(\bx_1,\ldots,\bx_s)= \frac{d!}{m_1!\cdots m_s!} \Gamma_f( \underbrace{\mathbf x_1,\ldots,\bx_1}_{m_1},\ldots,
 \underbrace{\bx_s,\ldots,\bx_s}_{m_s}),
\end{equation}
and every monomial of $R_{\boldsymbol{\alpha, \beta}}$ has degree strictly less than $m_j$ in at least one selected block $j$.
\end{Lemma}

\begin{proof}
Every monomial in every $F_{\boldsymbol\eta}$ has total degree $d$ in all coefficient blocks. A monomial having degree at least $m_j$ in every selected block has selected degree at least $m_1+\cdots+m_s=d$. Hence it has degree exactly $m_j$ in each selected
block and contains no unselected block. Its index is necessarily
\[
 m_1\boldsymbol\gamma_1+\cdots+m_s\boldsymbol\gamma_s=\boldsymbol\beta.
\]
Thus it comes only from $F_{\boldsymbol\beta}$. In the ordered sum \eqref{eq:F-beta}, the corresponding multiset occurs $d!/(m_1!\cdots m_s!)$ times. All other monomials are deficient in at least one selected block.
\end{proof}

\subsection{Application of the circle method} Set
\[
 C=\#\Delta_{r-1}(de)=\binom{de+r-1}{r-1}.
\]
For
$\boldsymbol\alpha=(\alpha_{\boldsymbol\beta})\in\mathbb T^C$, define
\begin{equation}\label{eq:S-def}
 S(\boldsymbol\alpha)= \sum_{\mathbf t\in\mathcal U} \psi\left( \sum_{\boldsymbol\beta\in\Delta_{r-1}(de)} \alpha_{\boldsymbol\beta}F_{\boldsymbol\beta}(\mathbf t)\right).
\end{equation}
By orthogonality \eqref{int-orth-id},
\begin{equation}\label{eq:N-integral}
 N_{r, q}(e)=\int_{\mathbb T^C}S(\boldsymbol\alpha)\,d\boldsymbol\alpha.
\end{equation}

For fixed $\boldsymbol\beta$, we apply Proposition \ref{prop:uniform} to the selected blocks in Lemma \ref{lem:isolation}, while fixing every nonselected block. Let $E_{\boldsymbol\beta}(\theta)$ denote the resulting fully multilinear sum. For this sum, the list of $d$ side lengths consists of $P_j$ repeated $m_j$ times, so its height sum is $H_{\boldsymbol\beta}$. Summing trivially over the fixed blocks gives
\begin{equation}\label{eq:S-one-beta}
 |S(\boldsymbol\alpha)|^{2^{d-1}} \leq \#\mathcal U^{2^{d-1}} E_{\boldsymbol\beta}(0)^{-1} |E_{\boldsymbol\beta}(\alpha_{\boldsymbol\beta})|.
\end{equation}
Varying $\boldsymbol\beta$ we deduce

\begin{equation}\label{eq:S-factor}
 |S(\boldsymbol\alpha)| \leq \#\mathcal U \prod_{\boldsymbol\beta} \left(\frac{|E_{\boldsymbol\beta}(\alpha_{\boldsymbol\beta})|} {E_{\boldsymbol\beta}(0)} \right)^\rho, \qquad \rho=\frac{1}{C2^{d-1}}.
\end{equation}

Suppose that $f$ is a nonsingular degree-$d$ form over $\F_q$ and
\[
 n>2^d(d-1)C.
\]
For every $\boldsymbol\beta$, Lemma \ref{lem:sigma} gives $\sigma_{\boldsymbol\beta}\geq n$. Since $\rho=1/(C2^{d-1})$, using the above inequality,
\[
 \sigma_{\boldsymbol\beta}\rho
 \geq n\rho>2(d-1).
\]

Next we apply the mean value estimate. The height condition of Proposition \ref{prop:mean-value} follows from Lemma \ref{lem:height-threshold}. Consequently
\begin{equation}\label{eq:Ebeta-mean}
 \int_\mathbb T |E_{\boldsymbol\beta}(\theta)|^\rho\,d\theta \leq E_{\boldsymbol\beta}(0)^\rho q^{-H_{\boldsymbol\beta}+d-1} \bigl(1+O(q^{-\delta})\bigr).
\end{equation}
The same positive $\delta$ may be used for all $\boldsymbol\beta$, since
there are only finitely many of them and all parameters are fixed.
By \eqref{eq:H-beta},
\[
 -H_{\boldsymbol\beta}+d-1=-h_d(\boldsymbol\beta).
\]

Integrating \eqref{eq:S-factor}, using Fubini, and applying \eqref{eq:Ebeta-mean} gives
\begin{align*}
 N_{r, q}(e)
 &\leq\int_{\mathbb T^C}|S(\boldsymbol\alpha)|\,d\boldsymbol\alpha\\
 &\leq \#\mathcal U
 \prod_{\boldsymbol\beta}
 E_{\boldsymbol\beta}(0)^{-\rho}
 \int_\mathbb T |E_{\boldsymbol\beta}(\theta)|^\rho\,d\theta\\
 &\leq \#\mathcal U
 \prod_{\boldsymbol\beta}q^{-h_d(\boldsymbol\beta)}
 \bigl(1+O(q^{-\delta})\bigr).
\end{align*}
Since $C$ is fixed, the product of the error factors is again $1+O(q^{-\delta})$. Using the identities in Lemma \ref{lem:lattice-identities} we have proved the following theorem.

\begin{theorem}\label{thm:finite-field}
Suppose that $f$ is a nonsingular degree $d$ form over $\F_q$ and
\[
 n>2^d(d-1)C.
\]
Then there is $\delta>0$, depending only on $n,d,e,r$, such that
\begin{equation}\label{eq:finite-estimate}
 N_{r, q}(e)
 \leq
 q^{n\binom{e+r}{r}-\binom{de+r}{r}}
 \bigl(1+O(q^{-\delta})\bigr),
\end{equation}
where the implied constant is independent of $q$.
\end{theorem}

\section{Proof of main results}\label{sec:geometric-deduction}

Let
\[
V_e := H^0\bigl(\mathbb P^r,\mathcal O_{\mathbb P^r}(e)\bigr), \qquad V_{de} := H^0\bigl(\mathbb P^r,\mathcal O_{\mathbb P^r}(de)\bigr).
\]
For convenience we set
\[
\widehat{\mu}_r(e)= {\mu}_r(e)+1.
\]
Suppose initially that $X $ is defined over a finite field $\mathbb F_q $. For every finite extension
\[
\mathbb F_{q^\ell}/\mathbb F_q,
\]
we put
\[
Y:= Q_e(X).
\]
Note that
\[
\#Y(\mathbb F_{q^\ell}) = \frac{N_{r,q^\ell}(e)-1}{q^\ell-1}.
\]
By Theorem \ref{thm:finite-field}, there exist constants
\[
\tilde C>0 \qquad\text{and}\qquad \delta>0,
\]
depending only on $n,d,e,r $, such that
\begin{equation}
N_{r,q^\ell}(e)
\leq
q^{\widehat{\mu}_r(e)\cdot \ell} + \tilde C\cdot q^{(\widehat{\mu}_r(e)-\delta)\cdot \ell}
\label{eq:N-upper-bound}
\end{equation}
for every $\ell \geq 1$. Therefore,
\begin{equation}
\limsup_{\ell\to\infty} q^{-\ell\mu_r(e)} \# Y(\mathbb F_{q^\ell}) \leq 1.
\label{eq:normalized-point-upper-bound}
\end{equation}

\subsection{Lower bound} The scheme $Y $ is a projective subscheme
\[
Y \subset \mathbb P\bigl(V_e^{\oplus n}\bigr) \simeq \mathbb P^{n\binom{e+r}{r}-1}
\]
defined by the coefficients of the degree $de $ form
\[
f(g_1,\ldots,g_n).
\]
Note that $Y$ is cut out by $\binom{de+r}{r}$ equations in $\mathbb P^{n\binom{e+r}{r}-1} $, every irreducible component of $Y$ has codimension at most $\binom{de+r}{r}$ using Krull's height theorem we get that each irreducible component of $Y$ has dimension at least $\mu_r(e)$.

\subsection{Lang--Weil estimate} Note that $Y$ is geometrically nonempty. Write
\[
Y_{\overline{\mathbb F}_q,\mathrm{red}} = \mathcal C_1\cup\cdots\cup \mathcal C_\lambda
\]
for the decomposition into geometrically irreducible components, and put
\[
\kappa := \max_{\lambda}\dim \mathcal C_\lambda.
\]
Let $c $ denote the number of components of dimension $\kappa $.

After replacing $\mathbb F_q $ by a finite extension, we may assume that every $\mathcal C_\lambda $ is defined over the base field and is geometrically irreducible.

For each component $\mathcal C_\lambda $ of dimension $\kappa $, the Lang--Weil estimate
gives
\[
\#\mathcal C_\lambda(\mathbb F_{q^\ell}) = q^{\kappa\cdot \ell} + O\bigl(q^{\ell\cdot(\kappa-\frac12)}\bigr)
\]
for $\ell \geq 1$.

If $\mathcal C_\lambda $ and $\mathcal C_{\lambda'} $ are distinct components of dimension $\kappa $, then
\[
\dim(\mathcal C_\lambda\cap \mathcal C_{\lambda'}) \leq \kappa-1.
\]
Indeed, an irreducible component of $\mathcal C_\lambda\cap \mathcal C_{\lambda'} $ of dimension $\kappa $ would be dense in both $\mathcal C_\lambda $ and $\mathcal C_{\lambda'} $, forcing $\mathcal C_{\lambda}=\mathcal C_{\lambda'} $.

The union of the components of dimension strictly less than $\kappa $, together with the pairwise intersections of the $\kappa $ dimensional components, has dimension at most $\kappa-1 $. Hence this exceptional locus has
\[
O(q^{\ell\cdot(\kappa-1)})
\]
 $\F_{q^\ell} $ points. It follows that
\begin{equation}
\#Y(\mathbb F_{q^\ell}) = c\cdot q^{\ell\cdot \kappa} + O\bigl(q^{\ell\cdot(\kappa-\frac12)}\bigr).
\label{eq:lang-weil-components}
\end{equation}
Therefore
\[
\kappa \leq\mu_r(e).
\]
On the other hand, every geometrically irreducible component of $Y $ has dimension at least $\mu_r(e) $. It follows that
\[
\kappa=\mu_r(e),
\]
and every geometrically irreducible component of $Y $ has dimension exactly $\mu_r(e) $.

Substituting $\kappa =\mu_r(e) $ into
\eqref{eq:lang-weil-components}, we obtain
\[
q^{-\ell\cdot \mu_r(e)} \#Y(\mathbb F_{q^\ell})= c + O(q^{-\frac\ell2}).
\]
Comparing \eqref{eq:normalized-point-upper-bound} and \eqref{eq:lang-weil-components} gives $c\leq 1$. Since $Y\neq \varnothing$, we have $c\geq 1$, and hence $c=1$. We conclude that
\[
Q_e(X)
\]
is geometrically irreducible and has dimension
\[
\mu_r(e).
\]

\subsection{Existence of $r$-planes} We verify that $\op{Mor}_e(\mathbb P^r, X)\neq \varnothing$.

\begin{Lemma}\label{lem:linear-plane}
Under the assumption
\[
 n> 2^d(d-1)\binom{de+r-1}{r-1},
\]
the hypersurface $X$ contains a linear $\mathbb P^r$.
\end{Lemma}

\begin{proof}
We work over the algebraic closure and use an inductive approach. Start with a point of $X$ and suppose inductively that $L=\mathbb P(W)\simeq \mathbb P^k$ is contained in $X$, where $0\leq k<r$ and $\dim W=k+1$. We want a vector $y\notin W$ such that $\mathbb P(W+Ky)\subset X$.
Because $f|_W=0$, the required condition is that all coefficients of positive powers of $\lambda$ in $f(w+\lambda y)$ vanish identically as functions of $w\in W$. For $1\leq j\leq d$, the coefficient of $\lambda^j$ is a form
of degree $d-j$ in $w$. Thus the number of homogeneous equations imposed on the class $[y]\in\mathbb P(K^n/W)\simeq \mathbb P^{n-k-2}$ is
\begin{equation}\label{eq:extension-equations}
 \sum_{j=1}^d\binom{d-j+k}{k}=\binom{d+k}{k+1}.
\end{equation}
If
\begin{equation}\label{eq:extension-condition}
 n-k-2\geq\binom{d+k}{k+1},
\end{equation}
then the projective dimension theorem ensures that these equations have a
common zero in $\mathbb P^{n-k-2}$, and this zero gives the required extension.

It remains to check \eqref{eq:extension-condition}. Since $e\geq1$ and
$k\leq r-1$,
\[
 \binom{de+r-1}{r-1}\geq\binom{d+r-1}{r-1}.
\]
Moreover
\[
 \binom{d+r-1}{r-1}=\frac r d\binom{d+r-1}{r} \geq\frac r d\binom{d+k}{k+1}.
\]
For $r\geq2$ and $d\geq2$, $2^d(d-1)r/d\geq4$. Hence the bound on degree $d$ implies,
\[
 n>\binom{d+k}{k+1}+k+2.
\]
Thus \eqref{eq:extension-condition} holds. Induction constructs a linear $\mathbb P^r\subset X$.
\end{proof}

Composing the inclusion $\mathbb P^r\hookrightarrow X$ with
\[
 [x_0:\cdots:x_r]\longmapsto[x_0^e:\cdots:x_r^e]
\]
produces a point of $\op{Mor}_e(\mathbb P^r,X)$.

\begin{proof}[Proof of Theorem \ref{thm:Q_e}]
The above discussion completes the proof when the base field is algebraic closure of a finite field.

We treat the general case by spreading out, see \cite[Section 2]{Browning2017}, \cite[Section 3]{browning2025rationalsurfaceslowdegree} for details. Suppose the base field has characteristic zero, or it is an arbitrary algebraically closed field of characteristic $p>d$. Spreading out $f$, $X$, and the coefficient scheme over an integral finitely generated subring of the base field, the preceding argument then applies to every closed fiber after shrinking the base. If the geometric generic fiber were reducible or had dimension different from $\mu_r(e)$, then after a finite extension of the generic point its components and their dimensions would spread out over a nonempty open subset, contradicting the finite field fibers. Thus the geometric generic fiber is irreducible of the required dimension. Base change to the given algebraically closed field completes the proof of Theorem \ref{thm:Q_e}.
\end{proof}

\begin{proof}[Proof of Theorem \ref{thm:main}]
 The morphism scheme is an open locus
\[
\op{Mor}_e(\mathbb P^r,X)\subset Q_e(X).
\]
By Lemma \ref{lem:linear-plane}, $\op{Mor}_e(\mathbb P^r,X) \neq \varnothing$ in our range and Theorem \ref{thm:Q_e} shows that $Q_e(X)$ is irreducible of the expected dimension, hence $\op{Mor}_e(\mathbb P^r,X)$ is irreducible of dimension $\mu_r(e)$. 
\end{proof}

\begin{remark}\label{rem:singular-targets}
The smoothness hypothesis is used in Lemma \ref{lem:sigma} to bound the codimension of the multilinear singular locus. More generally, let
\[
\mathfrak B(f):=\operatorname{codim}_{\mathbb A^n} \{x\in\mathbb A^n:\nabla f(x)=0\}.
\]
The same diagonal argument gives $\sigma\geq \mathfrak B(f)$. Consequently, the analytic estimate underlying Theorem \ref{thm:finite-field} continues to apply under the condition
\[
\mathfrak B(f)> 2^d(d-1)\binom{de+r-1}{r-1}.
\]
We restrict the statements of the main theorems to smooth hypersurfaces in order to keep their geometric formulation transparent.
\end{remark}

\bibliographystyle{alphaurl}
\bibliography{references}

\end{document}